\documentclass[12pt]{amsart}
\usepackage{amssymb,amsfonts,amsthm, color}
\usepackage{graphicx,subfloat}
\usepackage{bm}
\newtheorem{thm}{Theorem}[section]
\newtheorem{cor}[thm]{Corollary}
\newtheorem{lem}[thm]{Lemma}

\newtheorem{pro}[thm]{Proposition}

\theoremstyle{definition}

\theoremstyle{remark}

\numberwithin{equation}{section}

\title[Polynomial convergence to equilibrium]{Polynomial convergence
  to equilibrium for 
   a system of interacting particles}

\author{Yao Li}
\address{Yao Li: Department of Mathematics and
    Statistics, University of Massachusetts Amherst, Amherst, MA, 01003}
\email{yaoli@math.umass.edu}

\author{Lai-Sang Young}
\address{Lai-Sang Young: Courant Institute of Mathematical
Sciences, New York University, New York, NY 10012, USA}
\email{lsy@cims.nyu.edu}
\thanks{LSY was supported in part by NSF Grant DMS-1363161. }

\begin{document}
\maketitle

\begin{abstract} We consider a stochastic particle system in which a finite
number of particles interact with one another via a common energy tank.
Interaction rate for each particle is proportional to the square root of its kinetic
energy, as is consistent with analogous mechanical models. Our main result is
that the rate of convergence to equilibrium for such a system is $\sim t^{-2}$, more
precisely it is faster than a constant times $t^{-2+\varepsilon}$ for any $\varepsilon>0$. A discussion of exponential
{\it vs} polynomial convergence for similar particle systems is included.
  \end{abstract}

\vskip .3in
This paper is about dynamical models of (large numbers of) interacting 
particles, a topic of fundamental
importance in both dynamical systems and statistical mechanics. 
Our focus is on the speed of convergence to equilibrium, equivalently 
the rate of decay of time correlations. On a fixed energy surface, Liouville
measure, which describes the states of a system in equilibrium, 
does not depend on the dynamics generated by the Hamiltonian, 
but once the system is taken {\it out of equilibrium}, the speed with which 
it returns to equilibrium can be affected by dynamical details. One of the 
purposes of this paper is 
to call attention to the fact that for particle systems, this convergence can be fast
or slow depending on how the particles interact. 

While Hamiltonian models are considered to be physically more realistic than 
stochastic ones, questions of ergodicity and mixing for general Hamiltonian systems 
are out of reach at the present time, let alone the rate of mixing. 
Simplifications on the level of modeling are necessary if one is to gain insight into 
the problem. Since chaotic dynamics are known to produce
statistics very similar to those of genuinely random stochastic processes 
\cite{sinai1970dynamical, bowen1975equilibrium, young1998statistical, chernov2000decay, rey2008large}, it seems logical to first tackle stochastic models designed to capture similar underlying phenomena. 

The following model of binary collisions introduced by Kac \cite{kac1954foundations}
half a century ago as an idealization of Boltzmann dynamics was in this spirit.
In Kac's model, the velocities 
of $N$ particles are described (abstractly) by $N$ real numbers 
$v_1, v_2, \dots, v_N$, so that the system has total energy 
$\sum_{i=1}^N v_i^2 =E$. An exponential clock rings with rate $N$.
When it rings, a pair of particles, $i$ and $j$, is randomly chosen 
and assumed to interact, resulting in new velocities, $v'_i$ and $v'_j$, given by
\begin{eqnarray*}
v'_i & = & (\cos \theta) v_i - (\sin \theta) v_j\\
v'_j & = & (\sin \theta) v_i + (\cos \theta) v_j
\end{eqnarray*}
where $\theta \in [0,2\pi)$ is uniformly distributed.
This model has been much studied. Among other things, it has been shown
that its infinitesimal generator has a spectral gap uniformly bounded  
away from zero in size for all $N$ \cite{janvresse2001spectral,
  carlen2003determination, maslen2003eigenvalues}.
%Somewhat questionable in the model above is the constant rate of interaction 
%between all pairs of particles, a point acknowledged in \cite{kac1954foundations}.
Models with energy-dependent interactions, which are more realistic than the
constant rate of interaction in the original model, have also been
studied \cite{carlen2014spectral} , as have other variants of this
model; see e.g. \cite{sasada2013spectral, grigo2012mixing} for binary
collision processes on lattices and \cite{gaveau1986probabilistic} for
extensions to quantum N-body problems.  

In general, for systems with direct particle-particle interactions and an interaction potential 
that falls off with distance, it is very difficult to identify a simple stochastic rule that 
captures faithfully the deterministic dynamics.
In this paper, we consider a class of particle systems for which such modeling is 
more straightforward, namely when the particles do not interact with one another directly 
but only via their ``environment", or a ``hub".  
Concrete examples of mechanical models of this type were introduced 
in \cite{mejia2001coupled, rateitschak2000thermostating} and studied
later in \cite{eckmann2006nonequilibrium, lin2010nonequilibrium,
  li2014nonequilibrium, eckmann2007controllability,
  eckmann2006thermal, eckmann2006memory, yarmola2011ergodicity, konstantin2013ergodic}. In these models, the ``environment" is symbolized 
by the kinetic energy stored in rotating disks placed at various locations
in the physical domain. When a particle collides with a disk, 
energy is exchanged in accordance with a rule consistent with energy and angular 
momentum conservation; point particles do not ``see" each other otherwise. 
See Fig. 2.
The models considered in the present paper are a stochastic version of these 
mechanical models; details are given in Sects. 1.1 and 1.2. 

An example of the type of stochastic modification we make is that we ``forget"
the precise location of a particle, and replace the time to its next collision by
an exponential random variable with mean $\propto \frac{1}{\sqrt{e}}$ where $e$ is 
the kinetic energy of the particle. This idea was also used in 
\cite{eckmann2006nonequilibrium}, and is consistent with the statistics produced by 
chaotic dynamical systems. More detailed justification is given in Sect. 1.1.

We prove for our models that the speed of convergence to equilibrium is not exponential
but polynomial. More precisely, we show that for any $\gamma >0$,
this rate is faster than $\sim t^{\gamma-2}$. Because 
the rate of interaction is $\propto \sqrt{e}$, it is not hard to see that convergence
rate cannot be faster than $\sim t^{-2}$. Thus our results 
are sharp, and to our knowledge they are new; a literature search 
has not turned up comparable results involving polynomial rates of convergence. 
The closest that we are aware of are \cite{yarmola2013sub, yarmola2014sub}, 
which showed slower than exponential convergence for certain mechanical models 
with special properties (e.g. particles interacting only with heat baths, or 
particle systems on physical domains with special geometry). 

The speed of convergence to equilibrium, equivalently the rate of decay of time correlations, impacts the type of probabilistic limit laws obeyed 
by the system. We do not pursue that here as these questions will take us too
far afield, but remark only on some immediate consequences: 
With polynomial correlation decay, one cannot expect to have a  large deviation 
principle with a reasonable rate function \cite{wu2001large, wu2000uniformly,
    kontoyiannis2005large, balaji2000multiplicative}. As 
a result, Gallavotti-Cohen type fluctuation theorems will not hold \cite{rey2002fluctuations, lebowitz1999gallavotti}. A Markov chain 
central limit theorem for bounded observables, on the other hand, follows 
from polynomial ergodicity; see Theorem 5.

The main ideas of our proof are as follows: Since low-energy particles
are the source of slow convergence, we call a state of the system, equivalently
an energy configuration, ``active" if
every particle carries an energy above a certain minimum. Starting from
the set of active states we prove a Doeblin-type condition,  suggesting 
exponential correlation decay for an induced process. 
We then return to the full system, and propose to view the dynamics as having 
been refreshed, or renewed, each time a trajectory returns to the set of active states. 
This puts us in a framework bearing some resemblance to renewal
processes, for which it has been shown that the speed of convergence to 
equilibrium is determined 
by the moments of renewal times. Following ideas from renewal theory,
we seek to control 
first passage times to the set of active states. This is done by constructing 
a suitable Lyapunov function; see Section 2.

\medskip
\noindent {\it Polynomial vs exponential convergence: further examples.}  
The root cause of the slow convergence in our model is that once 
a particle acquires a low energy in an interaction, it simply stays 
``frozen" until its clock rings again; there is no way to activate 
it sooner. This need not be the case in models with direct particle-particle 
interactions, if another particle can pass by 
and activate a slow particle. The question of exponential 
vs polynomial rates of convergence to equilibrium is most transparent in
the setting of {\it one particle per site},  {\it nearest-neighbor interactions},
an example of which is the locally confined disk models introduced 
in \cite{bunimovich1992ergodic} 
and studied in \cite{gaspard2008heat, gaspard2008heat2}: 
A linear chain of cells 
is connected by openings. Inside each cell is 
a single finite-size convex body (hard disk), 
the diameter of which exceeds that of the opening so it is trapped, 
but adjacent disks can meet and exchange energy; see Fig. 1. 
For these models, the rate of convergence hinges on 
whether a disk can be completely out of reach of its neighbors. 
When the openings are large enough, heuristic argument and numerical
simulations both give exponential convergence.
On the other hand, if the openings 
between cells are small enough that a disk can 
get entirely out of reach of its neighbors, then a phenomenon similar
to that in the present paper can occur: it is easy to 
{\it prove} that the rate of mixing cannot be faster than $t^{-2}$;
see \cite{li2015stochastic}, which contains also a numerical study confirming 
that the rate of mixing is $\sim t^{-2}$, and
the rate of interaction between disks with kinetic energies 
$e_{i}$ and $e_{i+1}$ can be approximated by $\sim
\sqrt{\min\{ e_{i}, e_{i+1}\}}$.

We comment on related works: In a nonrigorous derivation, 
\cite{gaspard2008derivation} argued for the same model 
that under certain assumptions, the rate of interaction between the $i$th and $(i+1)$st disks is $\sim \sqrt{e_i + e_{i+1}}$. Assuming this interaction rate,  \cite{sasada2013spectral,grigo2012mixing} proved exponential rates 
of convergence for stochastic versions of these models. To our knowledge, 
this interaction rate appears in a certain rare interaction limit (when the openings
between cells tend to zero), and involves a rescaling
of time. In the mechanical model above, without any rescaling of time,
it is a simple mathematical fact that correlations cannot decay faster than $t^{-2}$ 
when the disks can ``hide" from their neighbors. 

%
%
%In a nonrigorous derivation in \cite{gaspard2008derivation}, it was argued under 
%certain assumptions that the rate of interaction between 
%the $i$th and $(i+1)$st disks is $\sim \sqrt{e_i + e_{i+1}}$
%where $e_i$ and $e_{i+1}$ are the kinetic energies of the disks. 
%Assuming this interaction, exponential rates of convergence were proved
%for stochastic versions of these models in 
%\cite{sasada2013spectral, li2013existence}; see also \cite{grigo2012mixing}. 
%On the other hand, if the openings 
%between cells are small enough that a disk can 
%get entirely out of reach of its neighbors, then a phenomenon similar
%to that in the present paper can occur, namely that when a slow disk gets out of
%reach, the time to its next collision is bounded below by $\sim
%\frac{1}{\sqrt{e_i}}$. In the case of small openings, it is not hard to {\it prove}
%that the rate of mixing cannot be faster than $t^{-2}$. See \cite{li2015stochastic}, 
%which also contains numerical study confirming that the rate of mixing is in fact 
%$\sim t^{-2}$; furthermore, it was shown that
%the rate of interaction between disks with kinetic energy 
%$e_{i}$ and $e_{i+1}$ can be approximated by $\sim
%\sqrt{\min\{ e_{i}, e_{i+1}\}}$. We remark that the rate $\sim \sqrt{e_i + e_{i+1}}$ used in
%\cite{sasada2013spectral, grigo2012mixing} comes from taking a certain
%time rescaling limit, while the result in \cite{li2015stochastic} is
%obtained at the original time scale.   

\begin{figure}
\centerline{\includegraphics[width =
  \textwidth]{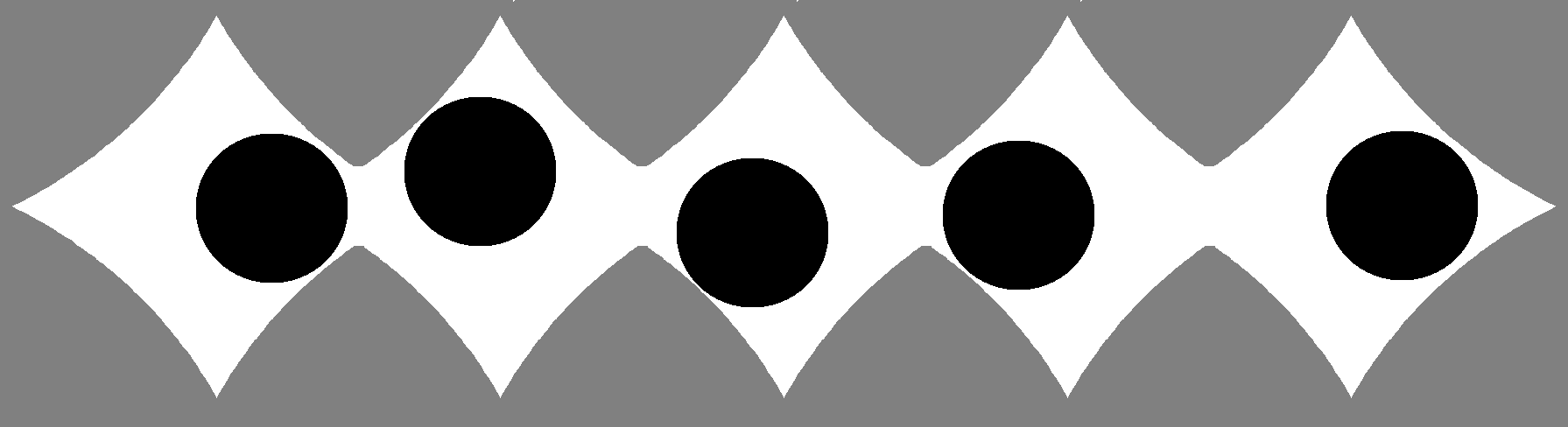}}
\caption{\small Locally confined hard disks model. Whether the system converges
to equilibrium at exponential or polynomial speeds depends on its
geometric configuration, specifically whether or not there are positions
where a disk (black) can be out of reach of  its neighbors.}
%  {\bf [Change to Figure 1]}}
\label{fig1b}
\end{figure}

\medskip \noindent {\it Organization of this paper:} Section 1 contains a precise
model description and statement of results. The bulk of the technical work 
goes into the construction of a Lyapunov function; this is carried 
out in Section 2.  In Section 3, we use this Lyapunov function to deduce
 the desired results on polynomial convergence to equilibrium.

\medskip
%%%%%%%%%%%%%%%%%%%%%%%%%%%%%%%%%%%%%%%
%%%%%%%%%%%%%%%%%%%%%%%%%%%%%%%%%%%%%%%%
\section{Model and results}

As explained in the Introduction, the models considered in this paper are stochastic
versions of some known mechanical models. We begin with a  review of these
mechanical models, followed by a discussion of the rationale for replacing
the deterministic dynamics by Markovian dynamics. Sect. 1.2 contains the precise definitions of the models
studied in the rest of this paper, and the statement of results are announced
in Sect. 1.3.

\medskip
\subsection{Mechanical models with particle-disk interactions} \ 

\medskip
We review here a class of models consisting of a rotating disk and a finite number
of  particles
in a closed domain, energy being exchanged when a particle collides with the disk.
The rules of energy exchange are borrowed from \cite{mejia2001coupled}; see also \cite{rateitschak2000thermostating}. 
These models, both in and out of equilibrium, were studied in \cite{eckmann2006nonequilibrium}.

A precise model description is as follows:
Let $\Gamma \subset \mathbb{R}^{2}$ be a bounded domain with concave
piecewise $C^{3}$ boundary; see Fig 2 for an example. 
In the interior of $\Gamma$ is a rotating disk $D$, nailed down at its center 
and rotating freely, carrying with it a finite amount
of kinetic energy. In the region $\Gamma \setminus D$ are
$m$ point particles,
each moving with uniform motion until it collides with $\partial \Gamma$ or $D$. 
Upon collision with $\partial \Gamma$, a particle is reflected elastically.
Upon collision with $D$, energy is exchanged according to the following rule:
Let $v$ be the velocity of the particle just prior to collision, $v=v_n + v_t$ its
decomposition into components that are normal and tangential to the disk,
and let $\omega$ denote the angular velocity of the disk. If $\ '$ denotes the
corresponding velocities following the collision, then from the conservation of
energy and angular momentum, one obtains, following \cite{mejia2001coupled},
\begin{eqnarray*}
v_n' & = & - v_n\\
v_t' & = & v_t - \frac{2\eta}{1+\eta} (v_t - R \omega)\\
R \omega' & = & R \omega + \frac{2}{1+\eta}(v_t - R \omega)\ .
\end{eqnarray*}
In these formulas, $\bar{m}$ is the mass of the particle, $R$ is the radius of the disk,  
$\theta$ is the moment of inertia of the disk, and $\eta = \theta/(\bar{m}R^2)$.
This is a complete description of the model. 

\begin{figure}[h]
\centerline{\includegraphics[width =
  0.65\textwidth]{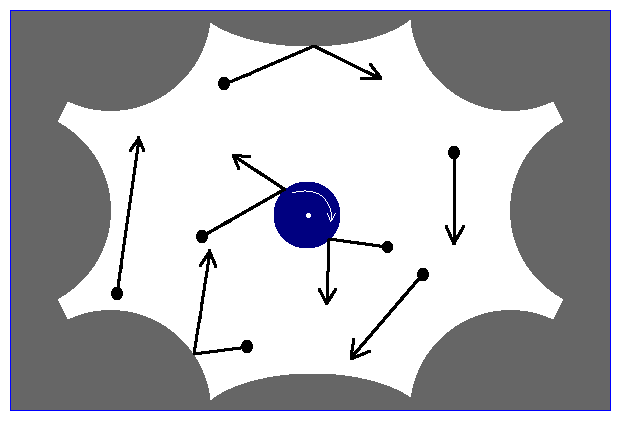}}
\label{fig1}
\caption{\small Example of a mechanical system that motivated the present study:
Particles in a domain $\Gamma $ (white) are scattered as they are reflected off
$\partial \Gamma$, and energy is exchanged when a particle
collides with the rotating disk (blue) nailed down at the center of the domain.
}
\end{figure}

Choosing $R=\eta=1$ leads to the especially simple equations
\begin{equation} \label{update}
v_n' \ = \ - v_n\ , \qquad v_t' \ = \ \omega\ , \qquad \omega' \ = \ v_t\ .
\end{equation}
For simplicity, we will work with these special parameters, though conceptually
it makes no difference in the present study.

\medskip \noindent
{\bf Connection to stochastic model}

\smallskip
Though easy to describe, an analysis of the mechanical model above is considerably
 outside of
the reach of current dynamical systems techniques. Thus we seek to simplify the model 
while retaining its essential characteristics, including the way in which energy is 
transferred among particles. By ``forgetting" the
precise locations of particles in the cell and their directions of travel, as well as
the direction of rotation of the disk, we turn the deterministic dynamical system above into a Markov process. Specifically,
the times to energy exchange for a particle are determined by exponential
distributions with mean $x^{-1/2}$ where $x$ is the instantaneous kinetic energy of the particle,
and the repartitioning of energy at exchanges are as in \eqref{update} assuming 
random angles of incidence. Details are given in Sect. 1.2.

We provide below some heuristic justification for the memory loss and interaction 
rates proposed in the last paragraph:

First we explain the rationale behind neglecting
 precise locations within a cell. Billard systems on domains 
with concave boundaries (or scatterers) are well known to exhibit chaotic, 
or hyperbolic, behavior \cite{sinai1970dynamical, chernov2006chaotic}. {\it Hyperbolicity} 
here refers to exponential divergence of nearby orbits, a property that leads 
to rapid loss of memory of trajectory history. 
By taking the rotating disk 
in our model to be relatively small, between energy exchanges
a typical particle trajectory is reflected 
many times as it bounces off the walls of the domain.
(Adding more scatterers 
in $\Gamma \setminus D$ as was done in \cite{li2014nonequilibrium} 
will further enhance mixing.)
As our system is a hyperbolic billiard between collisions with the rotating disk, 
the rapid loss of memory gives justification for neglecting precise locations 
within a cell.  

Next we explain the use of exponential random variables to describe the
times between collisions. Another well known fact for
strongly hyperbolic systems including billiards is that for points randomly distributed
in a specific region, return times to this region have exponentially small tails
\cite{young1998statistical}. Thus for particles that emerge from an energy
exchange with a fixed energy but randomly distributed otherwise in terms
of location and angle, we can expect the times to their next collision with
the disk to have
an exponentially small tail.  

Finally, fixing initial location and direction of travel, the time for a particle  to reach a 
pre-specified region is proportional to its speed; that is the rationale for assuming
mean collision time is proportional to $x^{-1/2}$. 

For another confirmation of the close connection between the stochastic model
in Sect. 1.2 and the mechanical model above, notice that modulo constants
their invariant measures coincide; see the remark following Proposition 1.

%A
%  further justification is that the invariant probability measure of
%  the mechanical model and that of its stochastic modification are
%  consistent. See the discussion after Proposition 1. 

%
%
%\vskip 1in
%by appealing to the exponential decay of
%correlations in billiard systems with concave domain boundaries
%\cite{young1998statistical, chernov2006chaotic}. By
%taking the disk to be relatively small so a typical particle trajectory is reflected 
%multiple times as a billiard trajectory,  the time to collision with the disk 
%and the angle of incidence in this collision will indeed be quite random and largely
%independent of history.
%If one so desires, adding more scatterers in $\Gamma \setminus D$ (as was done
%in \cite{li2014nonequilibrium}) will further enhance mixing. The discussion above
%applies to particle trajectories assuming they travel at identical speeds. 
%Here their speeds are unequal, and we observe that fixing initial
%location and direction of travel, the time to collision for a particle is proportional
%to its speed; that is the rationale for taking $x^{-1/2}$ to be the mean collision time.

\medskip
%%%%%%%%%%%%%%%%%%%%%%%%%%%%%%%%%
\subsection{Precise description of stochastic model} \ 

\medskip
The stochastic model considered in the rest of this paper is a time-homogeneous 
Markov jump process $\mathbf{x}_t, t \geq 0$, with
$$
\mathbf{x}_t = (x^1_t, \dots, x^m_t, y_t)\ .$$ 
Here $m$ is a fixed positive integer, $x^1_t, \dots, x^m_t$ 
are the energies of the $m$ particles at time $t$, and $y_t$ is the energy of the
disk, which we regard from here on as an abstract ``energy tank".
As the domain is assumed to be closed, 
total energy  remains  constant in time, i.e.,
there exists a constant $\bar E>0$ such that $\sum_i x^i_t + y_t = \bar E$ for all
$t \ge 0$. Thus the state space of $\mathbf{x}_{t}$ is the open $(m+1)$-dimensional
simplex
$$
{\bf \Delta} = {\bf \Delta}^{m+1}(\bar E) = \left\{ (x^{1}, \cdots, x^{m}, y) \in
  \mathbb{R}^{m+1}_{+} \,|\, y + \sum_{i =
    1}^{m} x^{i} = \bar{E}  \right\}\,.
$$

As in the mechanical model in Sect. 1.1, the particles in this system do not 
interact directly with one another. Instead, they interact via the energy tank, 
which symbolizes the ``environment" within the domain, and it is 
these particle-tank interactions that give rise to the jumps in the process.
The rules of interaction are as follows: 
Particle $i$ carries a clock that rings at an exponential rate equal to 
$\sqrt{x^{i}_t}$; notice that this rate changes each time the particle acquires a new energy.
The clocks carried by different particles are independent of
one another and of history. When its clock rings, a particle
exchanges energy with the tank according to the following rule:
Suppose the clock of particle $i$ rings at time $t$, and let 
$\mathbf{x}_{t^+}  =  (x^1_{t^+},  x^2_{t^+},  \dots,  x^m_{t^+},  y_{t^+})$
denote the state immediately following the interaction at time $t$.
Then assuming that the angles of incidence are uniformly distributed, 
the rules for updating, i.e. \eqref{update}, translate into
\begin{equation} \label{exch}
x^i_{t^+} = y_{t} + (1-u^2) x^i_{t}, \quad y_{t^+} = u^2 x^i_{t}, \qquad
\mbox{and } \quad x^j_{t^+} = x^j_{t} \quad \mbox{ for } j \ne i\ ,
\end{equation}
where $u \in (0,1)$ is a uniform random variable. For a detailed calculation,
see \cite{li2014nonequilibrium}.

The transition probabilities above together with an initial condition $\mathbf{x}_0$
defines the Markov process $\mathbf{x}_{t}$. The notation $\mathbf{x}_t = 
(x^1_t, \dots, x^m_t, y_t)$ is used throughout; in particular, $x^i$ is used exclusively
to denote the energy of the $i$th particle, not the $i$th power of $x$.
 
\medskip
We fix also the following notation: For $t \geq 0$ and $\mathbf{x} \in \mathbf{\Delta}$,
let $P^{t}(\mathbf{x},\cdot)$ be the transition probabilities of the process $\mathbf{x}_{t}$. 
That is to say, $P^{t}(\mathbf{x},\cdot)$ is the Borel probability distribution on $\mathbf{\Delta}$
describing the possible states of the system $t$ units of time later given that its 
initial condition is $\mathbf{x}$.
To simplify notation, we use the same notation for the left and right operators 
generated by $P^{t}$:
$$
  (P^{t}\xi)( \mathbf{x}) = \int_{\mathbf \Delta} P^{t}(\mathbf{x},
  \mathrm{d}\mathbf{y}) \xi(\mathbf{y}) 
$$
for a measurable function $\xi$ on $\mathbf{\Delta}$, and 
$$
  (\mu P^{t})(A) = \int_{\mathbf \Delta} P^{t}(\mathbf{x}, A)
  \mu (\mathrm{d}\mathbf{x}) 
$$
for a probability measure $\mu$ on $\mathbf{\Delta}$. Finally we say
$\mu$ is an invariant measure for the process $\mathbf x_t$ if 
$\mu P^t = \mu$ for all $t>0$.

\medskip

%%%%%%%%%%%%%%%%%%%%%%%%%%%%%%%%%%%%%%%
\subsection{Statement of results} \ 

\bigskip \noindent
{\bf Proposition 1. } {\it The probability measure $\pi$
with density
$$
  \rho(x^{1}, \cdots, x^{m}, y) = \frac{1}{Z} \ y^{-1/2} 
$$
where $Z$ is a normalizing constant is an invariant measure 
for the process $\mathbf x_t$.}

\bigskip
By the change of variables $x = |\mathbf{v}_{i}|^{2}$ and $y =
\tilde{\omega}^{2}$, one sees that $\pi$ coincides with Liouville measure
on a fixed energy shell for a Hamiltonian system with $H=|\mathbf{v}_{i}|^{2}
+\tilde{\omega}^{2}$. Here $\mathbf{v}_{i}$ is the velocity of the $i$th particle,
and $\tilde{\omega}$ is the angular velocity of the rotating disk.

\bigskip \noindent
{\bf Theorem 1 (Uniqueness of invariant measure).} {\it The measure $\pi$ 
in Proposition 1 is the unique invariant probability for $\mathbf x_t$; hence it is ergodic.}

\bigskip \noindent
{\bf Theorem 2 (Speed of convergence to equilibrium).} {\it For every $\mathbf{x} \in \mathbf{\Delta}$ and $\gamma > 0$, 
$$
  \lim_{t\rightarrow \infty} \ t^{2-\gamma} \ \| P^{t}( \mathbf{x}, \cdot)
  - \pi \|_{TV } = 0 
$$
where $\| \cdot \|_{TV}$ is the total variational norm.}

\bigskip
Theorem 2 is in fact deduced from Theorem 3 below. For ${  \delta} >0$, let
$\mathcal{M}_{\delta}$ be the collection of probability measures $\mu$ on
$\mathbf{\Delta}$ such that  
$$
  \int_{\bf \Delta} \left ( \sum_{k = 1}^{m} (x^{k})^{2{\delta}- 1} +
    y^{{\delta}- \frac{1}{2}} \right )
  \mu(\mathrm{d}\mathbf{x}) \ < \ \infty\ .
$$

\bigskip \noindent
{\bf Theorem 3 (Polynomial contraction of Markov operator).} {\it For any $\gamma>0$ and 
$\mu, \nu \in \mathcal{M}_{\gamma/8}$, }
$$
  \lim_{t \rightarrow \infty} \ t^{2-\gamma} \ \| \mu P^{t} - \nu P^{t}\|_{TV}
  = 0 \,.
$$

\medskip 
The following simple argument shows that the bound in Theorem 3 is tight: 
Consider, for example, two initial distributions $\mu$ and $\nu$ that differ 
by a positive amount when restricted to the set $B_\epsilon:=\{x^i < \epsilon\}$ 
for some fixed $i$. For definiteness, let us assume that for all small enough 
$\epsilon$, 
$\mu|_{B_\epsilon} \le c\pi|_{B_\epsilon}$ and $\nu|_{B_\epsilon} \ge c'\pi|_{B_\epsilon}$ for some $c<1<c'$. Since $\pi(B_\epsilon) \propto \epsilon$, $x^i< \frac{1}{t^2}$ implies
that the probability with respect to $\pi$ of the $i$th clock ringing
before time $t$ is $< 1- e^{-1}$. It follows that 
$$
\| \mu P^{t} - \nu P^{t}\|_{TV} \ge \| (\mu P^{t} - \nu P^{t})|_{\{x^i < \frac{1}{t^2}\}} \|_{TV}
\ge \mbox{ constant} \cdot \frac{1}{t^2}\ .
$$

Another corollary of Theorem 3 is the rate of decay of time
correlations.

\bigskip \noindent
{\bf Theorem 4 (Polynomial correlation decay).} {\it For any $\gamma>0$ and 
$\mu \in \mathcal{M}_{\gamma/8}$, let $\xi$ and $\zeta \in L^{\infty}(\mathbf{\Delta})$. 
Then 
$$
  \left|\int_{ \mathbf{\Delta}} (P^{t} \zeta)(
  \mathbf{x}) \xi( \mathbf{x}) \mu( \mathrm{d}\mathbf{x}) - \int_{
    \mathbf{\Delta}} (P^{t}\zeta)( \mathbf{x}) \mu(
  \mathrm{d}\mathbf{x}) \int_{\mathbf{\Delta}} \xi( \mathbf{x}) \mu(
  \mathrm{d} \mathbf{x}) \right| \ = \ o \left(\frac{1}{t^{2-\gamma}} \right) 
  $$
as $t \to \infty$.}

\bigskip
The next result is another consequence of Theorem 3.

\bigskip
\noindent 
{\bf Theorem 5 (Central limit theorem). } {\it Let $f:
  \mathbf{\Delta} \rightarrow \mathbb{R}$ be a Borel function that is
uniformly bounded  $\pi$-a.s. For any $\delta>0$, let 
$\{f^{\delta}_{n}\}_{n = 1}^{\infty} = \{
f(\mathbf{x}_{0}), f(\mathbf{x}_{\delta}),\cdots, f(
\mathbf{x}_{n\delta}), \cdots\}$ be a sequence of observables,
and define 
$$
\bar{f} = \frac{1}{n}\sum_{i = 0}^{n} f^{\delta}_{n} \,.
$$
Then for any initial distribution $\mathbf{x}_{0}$, 
$$
  \sqrt{n}( \bar{f} - \mathbb{E}_{\pi}f ) \xrightarrow{d} N(0,
  \sigma^{2}_{f})   \qquad \mbox{ as } n \rightarrow \infty\ ,
$$
provided 
$$
  \sigma_{f}^{2} := \mathrm{var}_{\pi}\{ f( \mathbf{x}_{0}) \} + 2
  \sum_{i = 1}^{\infty} \mathrm{cov}\{ f(\mathbf{x}_{0}),
  f(\mathbf{x}_{i\delta})\} < \infty \,.
$$
}

%%%%%%%%%%%%%%%%%%%%%%%%%%%%%%%%%%%%%%%%
%%%%%%%%%%%%%%%%%%%%%%%%%%%%%%%%%%%%%%%%%
\section{Construction of Lyapunov function}

Let $\mathbf{\Delta}$ be as in Sect. 1.2. For $\alpha < \frac12$, we define
$V = V_{\alpha} : \mathbf{\Delta} \to \mathbb R^+$ by
$$
  V( \mathbf{x}) = V_{\alpha}( \mathbf{x}) = \sum_{i = 1}^{m} (x^{i})^{-2\alpha} + y^{-\alpha}\ .
$$
Our main technical result is the following:

\medskip
\begin{thm}
\label{lyapunov} For $\alpha < \frac12$ close enough to $\frac12$ and $h>0$ small enough, 
there exist $c_{0}, M>0$ depending on $\alpha$ and $h$
such that for $V=V_\alpha$ and $\beta = 1- (4\alpha)^{-1}$,
$$
  (P^{h}V)( \mathbf{x}) - V( \mathbf{x}) \leq -c_{0} V( \mathbf{x})^{\beta}
$$
for every $\mathbf{x}  \in \{V > M\}$.
\end{thm}

\smallskip
The motivation for this choice of Lyapunov function is as follows. As noted
in the Introduction, low energy particles are our main concern, for
they are not expected to interact for a long time, and that slows down the mixing
process. For this reason, a desirable Lyapunov function should satisfy 
$V( \mathbf{x}) \rightarrow \infty$ as $\mathbf{x} \rightarrow \partial \mathbf{\Delta}$.
We explain heuristically why one may expect something along the lines of  
$P^{h}V - V \sim - h V^{1/2}$, corresponding to $\alpha, \beta \approx 2$:
Assume $x^1 \ll 1$
  is the smallest particle energy. Then $V(\mathbf{x}) \sim
  (x^{1})^{-1}$. If the clock of particle $1$ rings on the time interval $[0,
  h)$ and $y$ is ``large'', then the expected drop of $V( \mathbf{x})$ following
  an interaction
  is $\sim (x^{1})^{-1} \sim V( \mathbf{x})$. But the probability that
  the clock of particle $1$ will ring exactly once before time $h$ is  
  $\sim h \sqrt{x^{1}}$. This
  means the expected drop of $V(\mathbf{x})$ is $\sim h (x^{1})^{-1/2}
  \sim h V^{1/2}$.

\bigskip
It is convenient to use the following equivalent description of $\Phi_{t}$: 
Starting from $t = 0$, 
a clock rings at time $\tau_{1}$ where $\tau_{1}$ is an exponential 
random variable with mean $(\sum_{i  =1}^{m}
\sqrt{x^{i}_{0}} )^{-1}$. When this clock rings, energy exchange takes
place between exactly one particle and the tank, and the probability that
particle $i$ is chosen is 
$$
  \frac{\sqrt{x^{i}_{0}}}{\sum_{i = 1}^{m} \sqrt{x^{i}_{0}}} \,.
$$
The rule of energy redistribution is determined by equation
\eqref{exch} as before, and this process is repeated, i.e., at time $\tau_2$,
an exponential random variable with mean $\left(\sum_{i  =1}^{m}
\sqrt{x^{i}_{\tau_1^+}}\right)^{-1}$, the clock rings again, and so on.

We begin with the following technical estimate:

\begin{lem}
There exist constants $\epsilon_{0} > 0$ and $c^* > 0$ such that 
$$
  \mathbb{E}[V( \mathbf{x}_{\tau_{1}^+}) \,|\,
 \mathbf{x}_0]  \leq V( \mathbf{x}_{0}) -\frac{c^*}{\sum_{i = 1}^{m} \sqrt{x^{i}_{0}}}
  V( \mathbf{x}_{0})^\beta
$$
for every $\mathbf{x}_0 \in B$, where 
$$
    B = \left \{ \mathbf{x} \in \mathbf{\Delta} \,|\, y < \epsilon_{0}, \ \mbox{ or } \ 
  x^{i} < 4^{-\frac{1}{2\alpha}} \epsilon_{0} \ \mbox{ for some } \  i
  \in\{ 1, \dots, m \} \right \}\ .
$$
\end{lem}

\begin{proof} By definition,
$$
\mathbb{E}[V( \mathbf{x}_{\tau_{1}^+})  \,|\,
 \mathbf{x}_0] \ = \ V( \mathbf{x}_{0}) +  \frac{1}{\sum_{i = 1}^{m} \sqrt{x^{i}_{0}}}  \sum_{i = 1}^{m} Q_{i}
$$
where
$$ 
Q_i \ = \ \sqrt{x^{i}_{0}} \left \{
    \int_{0}^{1} \left[ ( x_{0}^{i}(1 - u^{2}) + y_{0} )^{-2\alpha} +
    (x_{0}^{i} u^{2})^{-\alpha} \right] \mathrm{d}u - \left[(x_{0}^{i})^{-2\alpha} +
    y_{0}^{-\alpha} \right] \right \}\ ,
$$
i.e., we need to show $\sum Q_i \le -c^*V( \mathbf{x}_{0})^\beta$ for some $c^*>0$.
In the rest of the proof, we will omit the subscript $0$ in $\mathbf{x}_{0}, \ x^i_0$ 
and $y_0$, and write
$$
  C_{1} = \int_{0}^{1} (1 - u^{2})^{-2\alpha} \mathrm{d}u
  \qquad \mbox{ and } \qquad 
  C_{2} = \int_{0}^{1} u^{-2\alpha} \mathrm{d}u ,
$$
noting that $C_{1}, C_{2} < \infty$ for $\alpha <\frac{1}{2}$. We will use
many times the bound
\begin{equation} \label{qi}
Q_i \ \le \ \sqrt{x^{i}} \left \{ \min\{C_1 (x^i)^{-2\alpha}, y^{-2\alpha}\} + 
C_2 (x^i)^{-\alpha} - (x^i)^{-2\alpha} - y^{-\alpha} \right\}\ .
\end{equation} 
Without loss of generality assume 
$$x^{1} = \min_{1\leq i \leq m}
x^{i}\ .
$$
 Let $0 < \epsilon_{0} \ll \epsilon_{1} \ll \bar{E}$  be
two small numbers to be determined. We decompose $B$, the neighborhood
of $\partial {\bf \Delta}$ in the statement of the lemma, 
into three regions (see Fig 3) and analyze each one as follows: 

\begin{figure}[h]
\centerline{\includegraphics[width = 6cm]{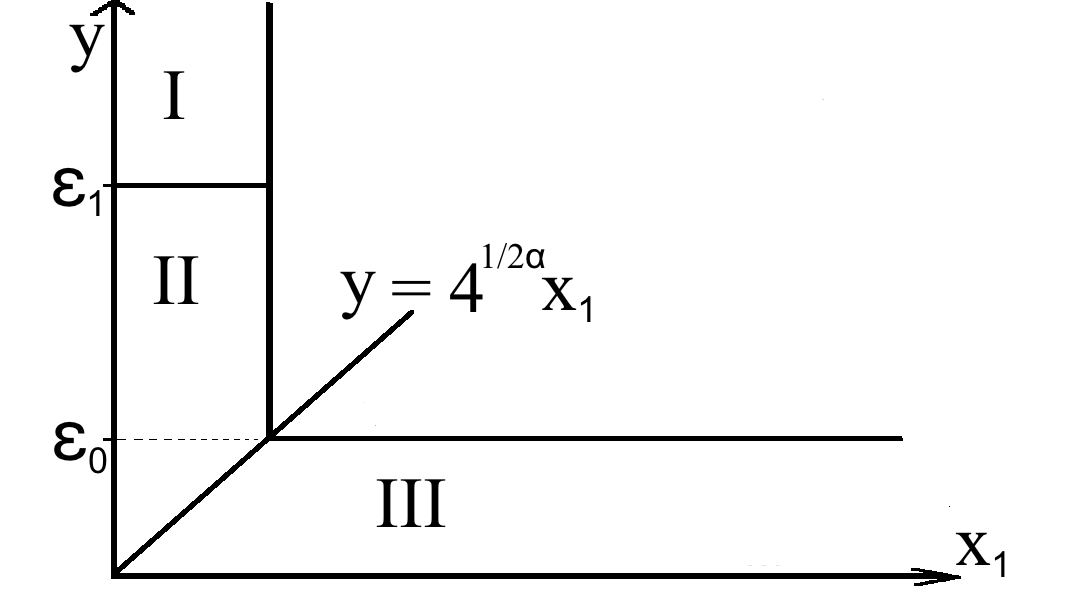}}
\caption{ Decomposition of neighborhood of $\partial \mathbf{\Delta}$}

\label{boundary}
\end{figure}

%{\bf [boundary region twice as wide, and larger fonts]}

\smallskip \noindent
{\bf Region I.}  \ $4^{\frac{1}{2 \alpha}} x^{1} < \epsilon_{0}, \ y \geq \epsilon_{1}$ 

\smallskip
With regard to lowering $V$, we clearly have the most to gain 
if particle 1 interacts with the tank: Applying \eqref{qi} to $x^{1}$ and substituting
in $y \geq \epsilon_{1}$, we obtain
\begin{eqnarray*}
 Q_{1}&\leq & \sqrt{x^{1}} \cdot \left\{  (\epsilon_{1})^{-2\alpha} + C_{2}
  (x^{1})^{-\alpha} - (x^{1})^{-2\alpha} \right \} \ .
%  \leq
%-\frac{1}{2}(x_{0}^{1})^{-2\alpha + \frac{1}{2}}
\end{eqnarray*}
Using $4^{\frac{1}{2 \alpha}} x^{1} <\epsilon_0 \ll \epsilon_1$, 
we see that the third term dominates. Hence
%Observe that the quantity inside the square brackets is $\ll (x_{0}^{1})^{-2\alpha}$:
%first, $\epsilon_1 \gg x_{0}^{1}$; second, $\epsilon_0$ is small enough that
%$(x_{0}^{1})^{-\alpha} \ll (x_{0}^{1})^{-2\alpha}$. Together this gives
$$
Q_1 \le -\frac{1}{2}(x^{1})^{-2\alpha + \frac{1}{2}}\ .
$$

For $i \ne 1$, we consider separately the following two cases:
For $x^{i} < \frac{1}{2} \epsilon_{1}<\frac{1}{2}y$, we have
\begin{equation} \label{x<y/2}
 Q_{i} \ \leq \  \sqrt{x^{i}} \left \{ (2x^i)^{-2\alpha} + C_{2}(x^{i})^{-\alpha} 
 - (x^{i})^{-2\alpha} \right \} \ ,
\end{equation}
which is $<0$ since the last term dominates.
If $x^{i}\geq \frac{1}{2} \epsilon_{1}$, then from \eqref{qi} we obtain
$$
Q_{i} \ \leq \ \sqrt{x^{i}} \{C_1(x^i)^{-2\alpha} + C_2 (x^i)^{-\alpha}\}
\le C' (\epsilon_1)^{- 2 \alpha + \frac12}
$$
for some $C'$ independent of $\epsilon_0$ or $\epsilon_1$. Notice
that we have used $\frac{1}{2} - 2 \alpha<0$, or $\alpha > \frac14$.

Altogether, we have shown, using $\epsilon_0 \ll \epsilon_1 \ll 1$, that
\begin{eqnarray*}
  \sum_{i = 1}^{m} Q_{i} & \leq & -\frac{1}{2} (x^{1})^{-2\alpha +
    \frac{1}{2}} + (m-1)C' (\epsilon_1)^{- 2 \alpha + \frac12}
    \leq -\frac{1}{3} (x^{1})^{-2\alpha + \frac{1}{2}}\ .
    \end{eqnarray*}
It follows from $V(\mathbf{x}) \le (m+1)(x^1)^{-2\alpha}$ that this is    
 $ \leq  - \frac{1}{3(m+1)} V(\mathbf{x})^\beta$.

\medskip \noindent
{\bf Region II.} \ $4^{\frac{1}{2 \alpha}} x^{1} < \epsilon_{0}, \ 4^{\frac{1}{2\alpha} } x^{1}
< y < \epsilon_{1}$

\smallskip

For $i=1$, applying \eqref{qi} and using $y > 4^{\frac{1}{2\alpha}} x^{1}$,
we obtain
\begin{eqnarray*}
 Q_{1}&\leq & \sqrt{x^{1}} \left\{ (4^{\frac{1}{2\alpha} } x^{1})^{-2\alpha}
 + C_{2} (x^{1})^{-\alpha} - (x^{1})^{-2\alpha}\right\}
\leq -\frac{1}{2}(x^{1})^{-2\alpha + \frac{1}{2}} \ .
\end{eqnarray*}

For $i \ne 1$, if $x^{i} < \frac{1}{2} y$, then the situation is as in (\ref{x<y/2}),
and $Q_i <0$.
The case where $x^{i} \ge \frac{1}{2} y$ is one of the more delicate: 
Applying \eqref{qi}, we obtain
$$
Q_{i} \ \leq \ C'' x^{- 2 \alpha + \frac12} - \sqrt{x} y^{-\alpha} \ \le
C''' y^{- 2 \alpha + \frac12} - \sqrt{x} y^{-\alpha}\ .
$$
Without loss of generality, assume $x^{2}, \cdots, x^{k} \ge \frac{1}{2}y$,
and $x^j < \frac{1}{2} y$ for all $j > k$. Then $\max \{x^2, \dots, x^k\} 
> \frac{\bar{E}}{2m}$. Therefore
\begin{eqnarray*}
\sum_{i = 2}^{k} Q_{i}& \leq  & (k-1) C''' y^{- 2 \alpha +\frac12} 
- \left(\sum_{i = 2}^{k} \sqrt{x^{i}} \right)y^{-\alpha}  \\
&\leq& y^{-\alpha}\left[ m C'''y^{-\alpha + \frac{1}{2}}
- \sqrt{\frac{ \bar{E}}{2m}}\right]\ .
\end{eqnarray*}
As $y< \epsilon_1$ and $\alpha < \frac12$, the quantity in square
brackets is $<0$ provided 
$\epsilon_1$ is sufficiently small. 

Thus arguing as in Region I, we have shown that
$$
\sum_{i = 1}^{m} Q_{i} \ \leq \ -\frac{1}{2} (x^{1})^{-2\alpha +
    \frac{1}{2}} \ < \  - \frac{1}{2(m+1)} V(\mathbf{x})^\beta\ .
    $$

\medskip \noindent
{\bf Region III.} \ $4^{\frac{1}{2 \alpha}} x^{1} \geq y,\  y < \epsilon_{0} $

\smallskip

Since $x^{i} \ge x^1 \geq 4^{-\frac{1}{2\alpha} } y$ for all $i$, a calculation
analogous to that in Region II gives
$$
\sum_{i = 1}^{m} Q_{i} \  \leq  \  y^{-\alpha}\left[ m C''''y^{ - \alpha +\frac12}
- \sqrt{\frac{ \bar{E}}{2m}}\right] \ < \ -\frac{1}{2}\sqrt{\frac{\bar{E}}{2m}} \ y^{-\alpha}
$$
provided $\epsilon_{0}$ is small enough. Since $V( \mathbf{x}) \leq m 
(x^{1})^{-2\alpha} + y^{-\alpha} \leq
(4m+1)y^{-2\alpha}$, it follows that 
$$
  \sum_{i = 1}^{m} Q_{i} \leq -\frac{1}{2}\sqrt{\frac{\bar{E}}{2m}}
  \cdot \frac{1}{\sqrt{4m+1}} V( \mathbf{x})^{\frac12} \leq 
  -\frac{1}{2}\sqrt{\frac{\bar{E}}{2m}}
  \cdot \frac{1}{\sqrt{4m+1}} V( \mathbf{x})^\beta
$$
since $\beta < \frac{1}{2}$. 

\medskip
The assertion is proved since it holds for $\mathbf{x}_0$ in all three
regions of $B$.
\end{proof}

 \medskip
\begin{proof}[Proof of Theorem 2.1]
Let $\tau_1 < \tau_2 < \dots$ be the times of clock rings as defined in
the paragraph preceding the statement of Lemma 2.2, and let $B$ be the 
neighborhood of $\partial \bf \Delta$ in Lemma 2.2. 
Letting $\tau_0=0$, we have shown that 
for any $n\ge 0$, if $\mathbf{x}_{\tau_n^+} \in B$, then
\begin{equation}\label{V1}
\mathbb{E}[V( \mathbf{x}_{\tau_{n+1}^+}) \,|\, \mathbf{x}_{\tau_n^+}] \ \le \ 
V(\mathbf{x}_{\tau_n^+})\ - \ \frac{c^*}{\sum_{i = 1}^{m} \sqrt{x^{i}_{\tau_n^+}}}
\  V( \mathbf{x}_{\tau_n^+})^\beta\ .
\end{equation}
For $\mathbf{x}_{\tau_{n}^+} \not \in B$, we will use the bound
\begin{equation} \label{V2}
\mathbb{E}[V( \mathbf{x}_{\tau_{n+1}^+}) \,|\, \mathbf{x}_{\tau_n^+}] \ \le \ M_0 + M_1
\end{equation}
where
$$M_{0} = \sup_{\mathbf{x} \in {\bf \Delta} \setminus B} V( \mathbf{x})
\quad \mbox{ and } \quad
  M_{1}  = \sup_{\mathbf{x} \in {\bf \Delta} \setminus B} \frac{1}{ \sum_{i = 1}^{m}
    \sqrt{x^{i}}} \sum_{i = 1}^{m} Q_{i}( \mathbf{x}) \ .
$$
It is easy to check that $M_{0}, M_{1} <\infty$. 

\medskip
We now use these estimates to deduce a bound for $P^{h}V$ for fixed $h>0$.
Let $S = \inf\{n, \tau_n >h\}$, and define $\hat \tau_n = \min\{\tau_n, \tau_{S-1}\}$.
Then
$$
 P^{h}V(\mathbf{x}) = \lim_{n\to \infty} \mathbb{E}[ V( \mathbf{x}_{\hat \tau_n^+})
 {\bf 1}_{S \le n+1} \,|\, \mathbf{x}_0 =\mathbf x ]
\le  \lim_{n\to \infty} \mathbb{E}[ V( \mathbf{x}_{\hat \tau_n^+})
 \,|\, \mathbf{x}_0 =\mathbf x ]\ .
$$
We will prove a uniform bound for 
$\mathbb{E}[ V( \mathbf{x}_{\hat \tau_n^+}) \,|\, \mathbf{x}_0 =\mathbf x ]$
for all $n \ge 1$.

First, assuming the worse of (\ref{V1}) and (\ref{V2}), we have 
\begin{equation} \label{V3}
\mathbb{E}[V( \mathbf{x}_{\tau_{n+1}^+}) \,|\, \tau_{n+1} \leq h] \ \le \ 
E[V(\mathbf{x}_{\tau_n^+})\,|\, \tau_{n+1} \leq h] + M_0 + M_1
\end{equation}
for every $n \ge 0$. Notice that conditioning on $\tau_{n+1}\leq h$ does not
affect the bounds in (\ref{V1}) and (\ref{V2}) because given 
$\mathbf{x}_{\tau_n^+}$, 
$\mathbf{x}_{\tau_{n+1}^+}$ is independent of $\tau_{n+1}-\tau_n$.
Second, as $\sum_{i = 1}^{m} \sqrt{x^{i}} \le \sqrt{m \bar{E}}$ for all $\mathbf x \in
\bf \Delta$, we have, for  every
$\mathbf{x}_{\tau_{n}^+} \in \bf \Delta$,
$$
\mathbb{P}[\tau_{n+1} \leq h\,|\, \mathbf{x}_{\tau_{n}^+}, \tau_n \leq h]
\leq \left (1 - e^{-h\sqrt{m \bar{E}}} \right ) \ 
\mathbb{P}[ \tau_n \leq h] \ ,
$$
so that inductively,
\begin{equation} \label{V4}
  \mathbb{P}[ \tau_{n+1} \leq h \,|\, \mathbf{x}_0] \leq \left (1 -
e^{-h\sqrt{m \bar{E}}} \right )^{n+1} \qquad \mbox{ for } \quad n \geq 0\ .
\end{equation}

The estimates (\ref{V3}) and (\ref{V4}) together imply the following: 
Given $\mathbf{x}_0=\mathbf{x}$, 
\begin{eqnarray*} 
& \ & \mathbb{E}[ V( \mathbf{x}_{\hat{\tau}_{n+1}^+})] \\ 
& = & \mathbb{E}[ V( \mathbf{x}_{\hat{\tau}_{n+1}^+}) \,|\, \tau_{n+1}>h] 
\cdot \mathbb P[\tau_{n+1}>h] + 
\mathbb{E}[ V( \mathbf{x}_{\hat{\tau}_{n+1}^+}) \,|\, \tau_{n+1}\leq h] 
\cdot \mathbb P[\tau_{n+1}\leq h]\\
& \le & \mathbb{E}[ V( \mathbf{x}_{\hat{\tau}_{n}^+}) \,|\, \tau_{n+1}>h] 
\cdot \mathbb P[\tau_{n+1}>h] \\
& & \hskip 1cm + 
\left(\mathbb{E}[ V( \mathbf{x}_{\hat{\tau}_{n}^+}) \,|\,
  \tau_{n+1}\leq h] +
M_0 +M_1\right) \cdot \mathbb P[\tau_{n+1} \leq h]\\
& \le & \mathbb{E}[ V( \mathbf{x}_{\hat{\tau}_n^+})] + 
(M_0 +M_1) (1-e^{-h\sqrt{m \bar{E}}})^{n+1}\ .
\end{eqnarray*}
Summing over $n$, this gives
$$
  P^{h}V( \mathbf{x}) \leq \mathbb{E}[V( \mathbf{x}_{\hat{\tau}_{1}^+}) \,|\,
\mathbf{x}_0=\mathbf x]  + \frac{M_{0} + M_{1}}{e^{-h\sqrt{m \bar{E}}}} \,.
$$

Let $h>0$ be small enough so that for all $\mathbf x \in \bf \Delta$, 
$$
 \mathbb{P}[ \tau_{1} \leq h \,|\, \mathbf{x}_0 = \mathbf x] = 1 - e^{-h \sum_{i = 1}^{m} \sqrt{x^{i}}} > \frac{h}{2} \sum_{i = 1}^{m} \sqrt{x^{i}}\ .
 $$
This is the only condition we impose on $h$. 

We choose $M'$ large enough so that $\{V>M'\} \subset B$, and
consider $\mathbf{x}_0 \in \{V>M'\}$. 
Noting again that $\mathbb{E}[V( \mathbf{x}_{{\tau}_{1}^+}) ]$ is independent of
$\tau_1$, we have, by Lemma 2.2, 
\begin{eqnarray*} 
\mathbb{E}[V( \mathbf{x}_{\hat{\tau}_{1}^+}) ] & = &
\mathbb{E}[V( \mathbf{x}_{\tau_{1}^+}) \,|\, \tau_1 \leq h] \cdot \mathbb P[\tau_1 \leq h] + 
V(\mathbf x_0) \cdot \mathbb P[\tau_1 > h]\\
& \le & \left(V(\mathbf{x}_0) - \frac{c^*}{\sum \sqrt{x^i}} V(\mathbf{x}_0)^\beta \right) \cdot \mathbb P[\tau_1 \leq h]
 + V(\mathbf{x}_0) \cdot \mathbb P[\tau_1 > h]\\
& \le & V(\mathbf{x}_0) - c^* \frac{h}{2} V(\mathbf{x}_0)^\beta\ .
\end{eqnarray*}
This gives
$$
P^{h}V( \mathbf{x}) \leq V(\mathbf{x}) - c^* \frac{h}{2} V(\mathbf{x})^\beta
+ (M_0+M_1) e^{h\sqrt{m\bar E}}\ .
$$
To complete the proof of Theorem 2.1, it suffices to replace $M'$ by 
a large enough number $M$ so that for $\mathbf x \in \{V>M\}$, the constant 
$(M_0+M_1) e^{h\sqrt{m\bar E}}$ is 
absorbed into $c_0 V(\mathbf x)^\beta$ for $c_0=c^* \frac{h}{4}$. 
\end{proof}

We record for later use the following fact that follows from the proof above:

\begin{cor} 
$$
\sup_{\mathbf{x} \not \in B} P^hV(\mathbf{x}) \ < \ \infty
$$
\end{cor}

\medskip
%%%%%%%%%%%%%%%%%%%%%%%%%%%%%%%%%%%%

\section{Completing the proofs}

After some preliminaries in Sect. 3.1, we proceed to the main
task of this section, the deduction of Theorem 3 from 
the Lyapunov function introduced. Two proofs are given, one in
Sects. 3.2 and 3.3 and the other in Sect. 3.4. Proofs of Theorems 2, 4 and 5 follow
quickly once Theorem 3 is proved.

\subsection{Existence and uniqueness of invariant measure} \ 

\medskip

\begin{proof}[Proof of Proposition 1]  Let $\pi$ be the probability measure with
density $\rho(x^1, \dots, x^m, y)$ $ =\frac{1}{Z} y^{-1/2}$. 
To prove $\pi = \pi P^{\xi}$ for $\xi \ll 1$,
it suffices to fix an arbitrary state $\mathbf{\bar x} = (\bar x^{1}, \cdots, \bar x^{m}, \bar y)
\in \mathbf{\Delta}$, let
$$
  D = D( \mathbf{\bar x}, \epsilon) = \{ \mathbf{x} \in \mathbf{\Delta} \,|\, 
  |x^i-\bar x^i|, |y-\bar y| < \epsilon \ \forall i \}
$$
for $\epsilon>0$ arbitrarily small, and show that 
$$
\mathbb{P}_\pi[\mathbf{x}_0 \in D, E] = \mathbb{P}_\pi [\mathbf{x}_\xi \in D, E]
+ O(\xi^{2})
$$
where $E$ is the event that exactly one interaction takes place on the interval $(0,\xi)$.
Clearly,
$$
\mathbb{P}_\pi[\mathbf{x}_0 \in {D}, E] = 
\xi (2\epsilon)^{m} \left(\sum_{i=1}^m \sqrt{\bar{x}^{i}} \right) \cdot Z^{-1}
  \bar{y}^{-1/2} + O(\xi^{2} \epsilon^m + \xi \epsilon^{m+1} ) \,.
$$

The estimation of $\mathbb{P}_\pi [\mathbf{x}_\xi \in D
, E]$ requires 
a straightforward computation identical to that in Lemma 6.6 
of \cite{li2014nonequilibrium}. 
\end{proof}

\smallskip
To prove uniqueness, we prove Doeblin's condition on a subset of $\bf \Delta$,
which for convenience we take to be a set of ``active states" of the form
$$
A_\epsilon :=\{\mathbf{x } \in \mathbf{\Delta} \,|\, x^i, y \ge \epsilon\}
$$
for some $\epsilon>0$. For $S \subset \bf \Delta$, let $U_{S}$ denote the uniform probability measure on $S$.

\smallskip
\begin{pro}
\label{uniformref} 
For any $t > 0$ and $\epsilon>0$, there exists a constant $\eta=\eta(\epsilon, t)$ 
such that for every 
$\mathbf{x} \in A=A_\epsilon$,
$$
  P^{t}(\mathbf{x}, \cdot) \geq \eta U_{A}(\cdot)\ .
$$
\end{pro}

\smallskip
\begin{proof} We cover $A$ with finitely many sets of the form 
$D=D( \mathbf{\bar x}, \xi)$ where $D( \mathbf{\bar x}, \xi)$
is as defined in the proof of Proposition 1 with the property that 
dist$(D, \partial {\bf \Delta}) > \frac{4\epsilon}{5}$. It suffices to show that
given any $t>0$, there exists $\eta>0$ such that for every 
$\mathbf{x} \in  A$, $P^{t}(\mathbf{x}, \cdot) \geq \eta U_D(\cdot)$
for all the $D$ in this cover. There are many ways to arrive at this outcome;
below we describe one possible scenario.

Let $\mathbf{x}$ and $D$ be fixed. There will be two rounds of interactions. 
The first round,
which takes place on the time interval $(0, \frac{t}{2})$, will result in most
of the energy collecting in the tank; and in the second round, which takes place
on $(\frac{t}{2},t)$, energy is redistributed according to $D$. In more detail,
starting from $\mathbf{x}$, the first round consists of particle 1 
interacting twice with the tank in quick succession, followed by particle 2,
and so on through particle $m$, with no other interactions besides these.
For each $i$, the goal of the second interaction is to result in 
$x^i_{\frac{t}{2}} \in (\frac{2\epsilon}{5} , \frac{3\epsilon }{5})$.
This requires two interactions to achieve because after the first
interaction, $x^i_{s^+} \ge y_s$ (see (\ref{exch})), and tank energy
prior to interaction with each particle is $ \ge \epsilon$. 
In the second round,
each particle interacts twice with the tank as before, resulting in
$x^i_t \in [\bar x^i-\xi, \bar x^i + \xi]$
uniformly distributed and independent of $x^j_t$ for $j=1,2,\dots, i-1$.

We leave it to the reader to check that the scenario above occurs with 
probability $\eta>0$
independent of $\mathbf{x}$ provided $\mathbf{x} \in A$.
\end{proof}

\smallskip
\begin{proof}[Proof of Theorem 1]  Let $A=A_\epsilon$ and $t$ be as above.
It is obvious that for any $\mathbf{x} \in {\bf \Delta}$, 
$P^{t/2}(\mathbf{x}, A)>0$. Together with Proposition 3.1, this implies that 
$P^{t}(\mathbf{x}, \cdot)$ has a strictly positive density on all of $A$, and 
that in turn implies that all $\mathbf{x} \in {\bf \Delta}$ belong in the same ergodic component, equivalently, $\mathbf{x}_t$ admits at most one invariant probability measure, 
which must therefore be $\pi$.
\end{proof}

\medskip
%%%%%%%%%%%%%%%%%%%%%%%%%%%%%%

\subsection{Review of tools from probability} \ 

\medskip
We recall here some tools that we will use to prove polynomial convergence.
As these are very general ideas, we will present them in the context of general Markov chains.
Let $\Psi_{n}$ be a (discrete-time) Markov chain on a measurable space
$(X, \mathcal{B})$ with transition kernels $\mathcal{P}(x, \cdot)$.

\bigskip \noindent
{\bf (A) Atoms of Markov chains.} A set $\alpha \in \mathcal{B}$ is called an {\it atom} 
if there is a probability measure $\theta$ on $(X, \mathcal{B})$ such that for all
$x \in \alpha$, $\mathcal{P}(x, \cdot) = \theta(\cdot)$. Most Markov chains on  
continuous or uncountable spaces do not possess atoms. 
We review here a technique introduced in \cite{nummelin1978splitting} which shows 
 that under
quite general conditions for $\Psi_n$, one can construct explicitly another chain,
$\tilde \Psi_n$, defined on an enlarged state space $(\tilde X, \mathcal{\tilde B})$,
such that  $\tilde \Psi_n$ is an extension of $\Psi_n$ and it has an atom. 

The relevant condition for $\Psi_n$ is that for some set $A_0 \in \mathcal{B}$, 
there exists a probability measure $\theta$ and a number $\eta>0$ such that
for every $x \in A_0$, $\mathcal{P}(x, \cdot) \ge \eta \theta(\cdot)$.
Let us call a set $A_{0}$ with this property a {\it special reference set}. 
  Assuming the existence of such an $A_0$, the splitting technique
of \cite{nummelin1978splitting} is as follows: Let $\tilde X = X \cup A_1$ (disjoint union) where $A_1$ is an
identical copy of $A_0$, with the obvious extension $\mathcal{\tilde B}$ of $\mathcal{B}$
to $\tilde X$. First we define the ``lift" of a measure $\mu$ on $(X, \mathcal{B})$ to 
a measure $\mu^*$ on $(\tilde X, \mathcal{\tilde B})$:
$$
  \left \{
\begin{array}{ll}
 & \mu^*|_X = (1 - \eta)\ \mu|_{A_0} + \mu|_{X \setminus A_0}\\
&  \mu^*|_{A_1} = \eta \ \mu|_{A_0}\ , \quad A_0 \cong A_1  \mbox{ via the natural
identification }.
\end{array}
\right .
$$
The transition kernels $\mathcal{\tilde P}(x, \cdot)$ are then given by
$$
  \left \{ 
\begin{array}{cl}
 \mathcal{\tilde P}(x, \cdot) = (\mathcal{P}(x, \cdot))^* &  x \in X \setminus A_0\\
 \mathcal{\tilde P}(\mathbf{x}, \cdot) = [(\mathcal{P}(x, \cdot))^{*}
 - \eta {\theta^{*}(\cdot)}]/(1 - \eta) 
 & x \in A_0 \\
\mathcal{\tilde P}( x, \cdot) = {\theta^{*}}(\cdot) & x \in A_1
\end{array}
\right .
$$
It is straightforward to check that the chain $\tilde \Psi_n$ projects to 
$\Psi_n$, meaning $(\mu \mathcal{P})^* = \mu^* \mathcal{\tilde P}$, so that
 $\|\mu \mathcal{P}^n - \nu \mathcal{P}^n\|_{\rm TV} \le 
\| \mu^* \mathcal{\tilde P}^n - \nu^* \mathcal{\tilde P}^n\|_{\rm TV}$.
Finally, $A_1$ is an atom for the chain $\tilde \Psi_n$ ---
this is the whole point of the construction.

%%%%%%%%%%%%%%%
\bigskip \noindent
{\bf (B) Connection to renewal processes.} For $E \in \mathcal{B}$, we let $\tau_E$
denote the first passage time to $E$, i.e.,
$$
\tau_E = \inf \{n>0 | \Psi_n \in E\}\ .
$$
Suppose the chain $\Psi_n$ has an atom $\alpha$, and that $\alpha$ is accessible,
 i.e.,
$\mathcal{P}_x [\tau_\alpha < \infty]=1$ for every $x \in X$. 
Given two initial distributions $\mu$ 
and $\nu$ on $X$, we wish to bound the rate at which
$\| \mu \mathcal{P}^n - \lambda \mathcal{P}^n\|_{\rm TV}$ tends to $0$
as $n \to \infty$ where $\|\cdot\|_{\rm TV}$ is the total variational norm.
One way to proceed is to run two independent copies of the chain
with initial distributions $\mu$ and $\nu$ respectively, and perform a coupling at
simultaneous returns to the atom $\alpha$. It is well known that if $T$ is the coupling time, then
\begin{equation} \label{eqcoupling}
\| \mu \mathcal{P}^n - \nu \mathcal{P}^n\|_{\rm TV} \le 2 \ \mathbb{P}[T>n]\ .
\end{equation}
The quantities $\mathbb{P}[T>n]$, on the other hand, can be studied via two associated
renewal processes as follows: 

Let $Y_0$ and $Y'_0$  be independent
$\mathbb N$-valued random variables having the distributions of $\tau_\alpha$,
the first passage time to $\alpha$, starting from 
$\mu$ and $\nu$ respectively, and let $Y_1, Y_2, \dots$ and $Y'_1, Y'_2, \dots$ 
be {\it i.i.d.} random variables the distributions of which are equal to that of 
$\tau_\alpha$ starting from $\alpha$. In addition, we assume the {\it return times}
to $\alpha$ are {\it aperiodic}, i.e.,
  $\mathrm{gcd}\{ n \ge 1 \,|\, \mathbb{P}[ Y_{i} = n ] > 0 \} = 1$. Then
$S_{n} := \sum_{i = 0}^{n}Y_{i}$ 
and $S'_{n} := \sum_{i = 0}^{n} Y'_{i}$, $n=0,1,2,\dots$, are renewal processes,
and $T$ above is the first simultaneous renewal time, i.e.
$$
  T = \inf_{n \geq 0} \{ S_{k_{1}} = S'_{k_{2}} = n \mbox{ for some }
  k_{1}, k_{2} \} \ .
$$
The following known result relates the finiteness of the moments of $T$ to
the corresponding moments for the distributions of $Y_0, Y'_0$ and $Y_1$:

\begin{thm}   
\label{coupling}
(Theorem 4.2 of \cite{lindvall2002lectures}) Let $Y_i$ and $Y'_i$ be as above.
Suppose that for some $\beta \ge 1$, we have
\begin{equation}\label{moments}
\mathbb{E}[Y_{0}^{\beta}], \quad \mathbb{E}[Y_{0}'^{\beta}]
\quad \mbox{ and } \quad \mathbb{E}[ Y_{1}^{\beta}] < \infty\ .
\end{equation}
Then $\mathbb{E}[T^{\beta}]$ is also finite.
\end{thm}

The discussion above implies the following:

\begin{cor} Let $\Psi_{n}$ be a Markov chain on $(X, \mathcal{B})$
with transition kernel $\mathcal{P}$. Suppose $\Psi_n$ has an atom $\alpha$ that is accessible and whose
return times are aperiodic. Let $\mu$ and $\nu$ be two probability distributions
on $X$, and assume that for some $\beta > 1$, 
$$
\mathbb E_\mu[\tau^\beta_\alpha], \ \ \mathbb E_\nu[\tau^\beta_\alpha]
\ \ \mbox{ and } \ \ \mathbb E_\alpha[\tau^\beta_\alpha] \ < \infty.
$$
Then 
$$
\lim_{n\to \infty}  n^\beta \| \mu \mathcal{P}^n - \nu \mathcal{P}^n\|_{\rm TV} = 0\ .
$$
\end{cor}

The proof is as discussed, together with the following general relation:  
Let $Z$ be a random variable taking values in 
$\mathbb{N}$, and let $\beta>1$. Then
\begin{equation}
  \label{moments}
\mathbb{E}[Z^\beta]<\infty \  \implies \ 
\lim_{n\to \infty} n^\beta \mathbb{P}[Z>n] = 0\ .
\end{equation}

%%%%%%%%%%%%%%%%%%
\bigskip \noindent
{\bf (C) Lyapunov function and moments of first passage times.} 
The following result, which is sufficient for our purposes,
 is a simple version of Theorem 3.6 of
\cite{jarner2002polynomial}:

\begin{thm} (Theorem 3.6 of \cite{jarner2002polynomial}) \label{thm36}
Let $\Psi_{n}$ be a Markov chain on $(X, \mathcal{B})$
with transition kernel $\mathcal{P}$. We assume that there exist a
function $W: X \rightarrow [1, \infty)$, a set $A \in \mathcal{B}$, constants 
$b, c > 0$ and $0 \leq \beta < 1$ such that
\begin{equation}
  \label{lyapunov1}
  \mathcal{P}W - W \leq - c W^{\beta} + b \mathbf{1}_{A} \,.
\end{equation}
Then there is a constant $\hat{c}$ such that for all $x \in X$,
$$
  \mathbb{E}_{x} \left [ \sum_{k = 0}^{ \tau_{A} - 1}
    (k+1)^{\hat \beta - 1} \right ] \leq \hat{c} W(x) \ , \qquad \hat \beta= (1 - \beta)^{-1}\ . 
$$
\end{thm}

Clearly, $\mathbb E_x[\tau_A^{\hat \beta}]$ is bounded above by a constant
times the expectation above.

The reader may notice that we have omitted some of the hypotheses in 
Theorem 3.6 of \cite{jarner2002polynomial} in the statement of Theorem \ref{thm36}
above. This is because they are not needed: here we consider only the first passage time
to $A$, which can be thought of as a set of the form $\{W \le {\rm constant}\}$,
while \cite{jarner2002polynomial} considers first passage times to arbitrary sets. We remark also that \cite{jarner2002polynomial} 
does not give the rate of convergence to equilibrium we claim; it shows that in general,
convergence rate is bounded by $\sim t^{\hat \beta - 1}$, 
but as we will see, additional information for our systems enables us to prove 
a faster convergence rate $\sim t^{\hat \beta - 2}$.

\bigskip \noindent
{\bf Remarks:} In {\bf (A)}, {\bf (B)} and {\bf (C)} above, we have outlined a general
strategy for deducing polynomial rates of convergence or of correlation decay
for Markov chains.
While we have cited specific references, they are not the only ones that contributed
to this general body of ideas 
  \cite{hairer2010convergence, nummelin1983rate, douc2009subgeometric,
  tuominen1994subgeometric, douc2004practical, fort2005subgeometric}.
We acknowledge in particular \cite{nummelin1983rate}, which was proved earlier
and which used similar ideas as above though some of the arguments
were carried out a little differently. We mention also \cite{young1999recurrence}, which models
deterministic dynamical systems with chaotic behavior as objects that are slight
generalizations of countable state Markov chains. This paper focuses on
tails of return times, i.e., $\mathbb P[\tau_\alpha > n]$, rather than on moments
of $\tau_\alpha$, to a set $\alpha$ that is effectively a special reference set as 
defined in {\bf (A)}; tails of first passage times
and moments are, as we have noted, essentially equivalent.

%%%%%%%%%%%%%%%%%%%%%%%%%%%%%%%%%%%%%%%%%%%%%%%
%%%%%%%%%%%%%%%%%%%%%%%%%%%%%%%%%%%%%%%%%%%%%
\subsection{Proofs of Theorems} \ 

\medskip
We first prove Theorem 3. Theorems 2, 4 and 5 follow easily;
their proofs are given at the end
of the subsection.

\medskip

Let $h>0$ be small enough for Theorem \ref{lyapunov}
to apply, and let
$$
  \hat{\mathbf{x}}_{n} = ( \hat{x}^{1}_{n}, \cdots, \hat{x}^{m}_{n},
 \hat y_{n}) = (x^{1}_{nh}, \cdots, x^{m}_{nh}, y_{nh} )\ , \qquad n = 0, 1, 2, \cdots\ ,
$$
be the time-$h$ sampling chain of $\mathbf{x}_{t}$.
Letting $\lfloor \frac{t}{h} \rfloor$ denote the largest integer $\le \frac{t}{h}$,
we observe that 
$$
\|\mu P^t - \nu P^t\|_{\rm TV} = 
\| (\mu P^{\lfloor  \frac{t}{h} \rfloor h}  - \nu P^{\lfloor
  \frac{t}{h} \rfloor h}) P^{(t-\lfloor  \frac{t}{h} \rfloor) h}\|_{\rm TV}  
\le \|\mu P^{\lfloor  \frac{t}{h} \rfloor h} - \nu P^{\lfloor
  \frac{t}{h} \rfloor h} \|_{\rm TV} \ ,
$$
so it suffices to prove the theorem for $\hat{\mathbf{x}}_{n}$ corresponding to a fixed $h$. From here on, $h$ is fixed, and since we will be working exclusively with the 
discrete-time chain $\hat{\mathbf{x}}_{n}$, the $\hat \ $ in $\hat{\mathbf{x}}_{n}$ is dropped
for notational simplicity.
 
Let $\gamma>0$ be  small enough that Theorem
\ref{lyapunov} applies with $\alpha = \frac12 - \frac{\gamma}{8}$.
We define
$$
  \mathcal{A} = \mathcal{A}_{\gamma, h} = \{ \mathbf{x}\in
  \mathbf{\Delta} \,|\, V_{\frac{1}{2} - \frac{\gamma}{8}}
  (\mathbf{x}) \leq M  \} 
$$
where $M = M( \frac{1}{2} - \frac{\gamma}{8}, h)$, and let
$$
{\tau}_{\mathcal{A}} = \inf_{n > 0} \{ {\mathbf{x}}_{n}
\in \mathcal{A} \}
$$
be the first passage time to $\mathcal{A}$. We plan to proceed as follows:

\noindent
(1) First we estimate the moments of ${\tau}_{\mathcal{A}}$ .

\noindent
(2) Using $\mathcal{A}$ as a special reference set, we split the chain,
obtaining an atom $\alpha$ for 

\ \ the split chain $\tilde{\mathbf{x}}_{n}$. 

\noindent
(3) We deduce from (1) the moments of $\tilde \tau_\alpha$,
the first passage time of $\tilde{\mathbf{x}}_{n}$ to $\alpha$, and 

\noindent
(4) finally, we apply Corollary 3.3 to $\tilde \tau_\alpha$ to obtain the desired results.

\smallskip

\begin{lem}
\label{taildiscrete} Given $\gamma$ as above, there exists $C =
C( \gamma )$ such that for all $\mathbf{x} \in \mathbf{\Delta}$,
$$
   \mathbb{E}_{ {\mathbf{x}}}[
   {\tau}_{\mathcal{A}}^{2 - \frac{\gamma}{2}} ] \leq
 {C} V_{\frac{1}{2} - \frac{\gamma}{8}}( \mathbf{{x}}) \ .
$$
\end{lem}

\begin{proof} We apply Theorem 2.1 to $V_{\frac{1}{2} -
    \frac{\gamma}{8}}$. From Corollary 2.3, it follows that if $W = \max\{ V_{\frac{1}{2} - \frac{\gamma}{8}}, 1\}$, then
$b := \sup_{\mathbf{x} \in \mathcal{A} } \left \{ P^{h} W (\mathbf{x}) - W( \mathbf{x}) \right \}  < \infty$, and we have
$$
  P^{h}W - W \leq - c W^{1 - \frac{2}{4 -
      \gamma}} + b \mathbf{1}_{\mathcal{A}} \,.
$$
Theorem \ref{thm36} then tells us that there is 
a constant $\hat{c}$ such that
for all $\mathbf{x}$,
$$
  \mathbb{E}_{{\mathbf{x}}}\left[ \sum_{k = 0}^{\tau_{\mathcal{A}} - 1} (k+1)^{1 -
    \frac{\gamma}{2}} \right] \leq \hat{c}W(\mathbf{{x}}) \,.
$$
As $W \leq   C_{2} V_{\frac{1}{2} -  \frac{\gamma}{8}} $ for some 
constant $C_{2} > 0$ that depends only on $\bar{E}, m$ and
$\gamma$, it follows that
$$
  \mathbb{E}_{\mathbf{{x}}} \left[ \tau_{\mathcal{A}}^{2 - \frac{\gamma}{2}}\right] \leq 2 \cdot
  \mathbb{E}_{\mathbf{{x}}}\left[ \sum_{k = 0}^{\tau_{\mathcal{A}} - 1} (k+1)^{1 - \frac{\gamma}{2}}\right] \leq
  2 C_{2} \cdot \hat{c} \cdot V_{\frac{1}{2} -  \frac{\gamma}{8}} ( \mathbf{{x}}) \,.
$$

This completes the proof.
\end{proof}

\medskip
Recall that for small $\delta>0$, $\mathcal{M}_\delta$ is the set of Borel probability
measures $\mu$ on $\mathbf{\Delta}$ such that
$$
  \int_{\mathbf{\Delta}}\left ( \sum_{k = 1}^{m}
     (x^{k})^{2\delta - 1 } + y^{\delta - \frac{1}{2}}\right ) \mu( \mathrm{d \mathbf{x}} )
     \ \equiv \ \int_{\mathbf{\Delta}} V_{\frac12 - \delta}(\mathbf{x}) \mu( \mathrm{d \mathbf{x}} )
\ < \    \infty \,.
$$

\smallskip

\begin{proof}[Proof of Theorem 3.]
Let $\gamma>0$ and $h > 0$ be as above,  and let
$\mu, \nu \in \mathcal{M}_{\gamma/8}$ be given.
It follows from Proposition 3.5 that 
$$
  \mathbb{E}_{\mu}[\tau_{\mathcal{A}}^{2 - \frac{\gamma}{2}} ] \ , \quad \mathbb{E}_{\nu}[
   \tau_{\mathcal{A}}^{2 - \frac{\gamma}{2}} ] < \infty \,.
$$

Observe next that $\mathcal{A}$ is a special reference set in the sense of 
Sect. 3.2{\bf (A)}; this follows from Proposition 3.1, for
$\mathcal{A} \subset A_{\epsilon}$ for $\epsilon>0$ small enough. 
We split the chain as discussed in Sect. 3.2{\bf (A)}, denoting the split chain 
by $\tilde{\mathbf{x}}_{n}$, and let $\mathcal{A}_{0}$ and $\mathcal{A}_{1}$ be 
identical copies of $\mathcal{A}$ in $\tilde{\mathbf{\Delta}}$, with $\mathcal{A}_{1} = \alpha$ being an atom.

To apply Corollary 3.3 to the chain $\tilde{\mathbf{x}}_{n}$, we first 
check that the atom $\alpha$ is accessible: It is easy to see that  if $\tilde{\tau}$ is
first passage time of $\tilde{\mathbf{x}}_{n}$, then $\tau_{\mathcal{A}} = \tilde{\tau}_{\mathcal{A}_{0} \cup \mathcal{A}_{1}}$,
and from Theorem 1, we know that $\mathcal{A}$ is accessible under $\mathbf{x}_{n}$.
Moreover, every time $\tilde{\mathbf{x}}_{n}$ returns to $\mathcal{A}_{0}\cup
\mathcal{A}_{1}$, it has probability $ \eta$ of entering $\alpha$. This guarantees
the accessibility of $\alpha$. Aperiodicity of return times to $\alpha$ follows
from the fact that for all $\tilde{\mathbf{x}}_0$, 
$\mathbb P[\tilde{\mathbf{x}}_1 \in \alpha] >0$.

It remains to pass the moments of $\tau_{\mathcal{A}}$ to the moments $\tilde{\tau}_{\alpha}$. For a measure $\lambda$ on $\mathbf{\Delta}$,  $\tilde \lambda$ denotes
its lift to $\tilde{\mathbf{\Delta}}$.

\smallskip
\begin{lem}
\label{lem31ofNT}
(i) $\mathbb{E}_{\alpha}[\tilde{\tau}_{\alpha}^{2 - \gamma}] < \infty $\ .

(ii) $\mathbb{E}_{\tilde{\lambda}}[\tilde{\tau}_{\alpha}^{2 - \gamma} ] < \infty $ 
for $\lambda$ with $\mathbb{E}_{\lambda}[ \tau_{\mathcal{A}}^{2 - \frac{\gamma}{2} }] <
\infty$ .
\end{lem}

\smallskip
This lemma follows from Lemma 3.1 of \cite{nummelin1983rate}; we provide
an elementary proof below for completeness. Assuming Lemma \ref{lem31ofNT}, 
we may now apply Corollary 3.3 to $\tilde{\mathbf{x}}_{n}$ with $\beta = 2-\gamma$,
giving a convergence rate of $t^{2-\gamma}$. To finish, recall from Sect. 3.2{\bf (A)}
that if
$\mathcal{P}$ and $\mathcal{\tilde P}$ are the Markov operators for $\mathbf{x}_{n}$ 
and $\tilde{\mathbf{x}}_{n}$ respectively, then
$\|\mu \mathcal{P}^n - \nu \mathcal{P}^n\|_{\rm TV} \le 
\| \tilde \mu \mathcal{\tilde P}^n - \tilde \nu \mathcal{\tilde P}^n\|_{\rm TV}$.
\end{proof}

\smallskip
\begin{proof}[Proof of Lemma \ref{lem31ofNT}]
To prove (i), it suffices to show that for some $\gamma'<\gamma$, there exists $C$
such that
\begin{equation} \label{tail}
  \mathbb{P}_{\alpha}[ \tilde{\tau}_{\alpha} > k ] \leq
  C k^{\gamma' - 2}  \qquad \mbox{ for  all } k\ .
\end{equation}
Let $\tau_n, n=1,2 \dots$, denote the $n$th entrance time into 
$\mathcal{A}_{0}\cup \mathcal{A}_{1}$, and let $\mathbf{N}$ be smallest
$n$ such that $\tilde \tau_\alpha = \tau_n$. Since at each $\tau_{n}$, the probability
of being in $\alpha$ is $\eta$, we have $\mathbb{P}[ \mathbf{N} > k]
= (1 - \eta)^{k}$. Note also that since 
$\sup _{\mathbf{x}
  \in \mathcal{A}} \mathbb{E}_{\mathbf{x}}[ \tau_{\mathcal{A}}^{2 -
  \frac{\gamma}{2} }] < \infty$ by Proposition 3.5, it follows that
$$
  \mathbb{P}[ \tau_{n+1} - \tau_{n} \geq k \,|\, \mathbf{N} > n,
  \tilde{\mathbf{x}}_{\tau_{n}}] \leq C' k^{-(2 - \frac{\gamma}{2})} 
$$
for some constant $C' $.

For any $\delta > 0$, we have
$$
  \{ \tau_{\mathbf{N}} > k^{1 + \delta} \} \subset \{ \mathbf{N} >
  k^{\delta} \} \cup \bigcup_{n = 0}^{\left \lfloor k^{\delta} \right
    \rfloor } \{ \tau_{n+1} - \tau_{n} > k ; \mathbf{N} > n \} \,.
$$
Thus
\begin{eqnarray*}
&&\mathbb{P}_{\alpha}[ \tilde \tau_{\alpha} > k^{1 + \delta} ]\\
&  \leq  &
\mathbb{P}_{\alpha}[ \mathbf{N} > k^{\delta} ] + \sum_{ n = 0}^{\left
    \lfloor k^{\delta} \right\rfloor } \mathbb{P}_{\alpha}[ \tau_{n+1}
- \tau_{n} > k\,|\, \mathbf{N} > n] \\
&=& \mathbb{P}_{\alpha}[ \mathbf{N} > k^{\delta} ] + \sum_{ n = 0}^{\left
    \lfloor k^{\delta} \right\rfloor } \int \mathbb{P}_{\alpha}[ \tau_{n+1}
- \tau_{n} > k\,|\, \mathbf{N} > n, \tilde{\mathbf{x}}_{\tau_{n}} = \mathbf{\tilde{x}}]
\mathbb{P}_{\alpha}[\tilde{\mathbf{x}}_{\tau_{n}} =
\mathrm{d}\mathbf{\tilde{x}}, \mathbf{N} > n ]\\
&\leq& (1 - \eta)^{k^{\delta}} + k^{\delta} C' k^{-(2 -
  \frac{\gamma}{2})} \,.
\end{eqnarray*}
Noting that the second term dominates for large $k$, we obtain (\ref{tail}) by choosing $\delta$ sufficiently
  small. 

The proof of (ii) follows similar steps and uses the finiteness of
$\mathbb{E}_{\mu}[ \tau_{\mathcal{A}}^{2 -
  \frac{\gamma}{2}}]$.
\end{proof}

%%%%%%%%%%%%%%%%%%%%%%%%%%%%%%%%%%%%
%\subsection{Proofs of Theorems 2 and 4}

\begin{proof}[Proof of Theorem 2.] 
A simple computation using the density of $\pi$
shows that $\pi \in \mathcal{M}_\delta$ for every $\delta > 0$. 
Also, for every $\mathbf{x} \in {\bf \Delta}$, the point mass $\delta_{\mathbf{x}}$
clearly belongs in $\mathcal{M}_\delta$ for all $\delta > 0$. 
Thus Theorem 2 is a special case of Theorem 3, with $\mu=\pi$ and 
$\nu=\delta_{\mathbf{x}}$.
\end{proof}

\begin{proof}[Proof of Theorem 4.] As a direct consequence of Theorem 3, we have
\begin{eqnarray*}
& &  \left|\int (P^{t} \zeta)(
  \mathbf{x}) \xi( \mathbf{x}) \mu( \mathrm{d}\mathbf{x}) - 
  \int (P^{t}\zeta)( \mathbf{x}) \mu(\mathrm{d}\mathbf{x}) \int \xi( \mathbf{x}) \mu(
  \mathrm{d} \mathbf{x}) \right| \\
& = & \left| \int \xi (\mathbf{x}) \left( (P^{t}\zeta)(\mathbf{x}) - 
\int (P^{t}\zeta)(\mathbf{z}) \mu(\mathrm{d}\mathbf{z}) \right) 
\mu( \mathrm{d} \mathbf{x}) \right|\\
%& \le & \|\xi\|_{L^{\infty}} \ \int \left( \int |\zeta| \ |\delta_{\mathbf{x}}P^t - \mu P^t| \right)\ \mu( \mathrm{d} \mathbf{x}) \\
& \le & \|\xi\|_{L^{\infty}} \ \|\zeta\|_{L^{\infty}} \ \int \|\delta_{\mathbf{x}}P^t - \mu P^t\|_{TV} \ 
\mu( \mathrm{d} \mathbf{x})\ = \ o \left(\frac{1}{t^{2-\gamma}} \right) \,.
%& \leq & \|\xi\|_{L^{\infty}} \ \|\zeta\|_{L^{\infty}} \left ( \int \|
%  \delta_{\mathbf{x}} - \pi \|_{TV} \mu( \mathrm{d} \mathbf{x}) + \|
%  \mu P^{t} - \pi \|_{TV} \right ) \\
%&=&  \|\xi\|_{L^{\infty}} \ \|\zeta\|_{L^{\infty}} \|\delta_{\mathbf{x}}P^t - \mu P^t\|_{TV} \ = \ o \left(\frac{1}{t^{2-\gamma}} \right) \,.
\end{eqnarray*}
\end{proof}

\begin{proof}[Proof of Theorem 5.]
Theorem 5 follows in a straightforward way from Theorem 3 and the Markov
chain central limit theorem (Corollary 2 of \cite{jones2004markov}).
It is a simple exercise to check that all conditions are satisfied by
the time-$\delta$ chain $\{\mathbf{x}_{n\delta}\}_{n = 0}^{\infty}$
for any $\delta > 0$.
\end{proof}

%
%\medskip
%{\it
%Suppose $\Psi_{n}$ is a Harris ergodic Markov chain on $(X,
%\mathcal{B})$ with transition kernel $\mathcal{P}$. Let $\mu$ be the stationary
%distribution of $\Psi_{n}$ and let $f: X\rightarrow \mathbb{R}$ be a Borel
%function that is uniformly bounded $\mu$-almost surely. Assume $\Psi_{n}$ is polynomial ergodic such that
%$$
%  \| \mathcal{P}^{n}(x, \cdot) - \mu \|_{TV} \leq M(x) n^{-m} \,,
%$$
%where $m > 1$ and $\mathbb{E}_{\pi}( M) < \infty$, then for any
%initial distribution, as $n \rightarrow \infty$,
%$$
%  \sqrt{n} ( \bar{f} - \mathbb{E}_{\mu}f ) \xrightarrow{d} N(0,
%  \sigma^{2}_{f}) \,,
%$$
%where 
%$$
%  \sigma_{f}^{2} := \mathrm{var}\{ f(\Psi_{0})\} + 2 \sum_{ i =
%    1}^{\infty} \mathrm{cov}\{ f( \Psi_{0}), f( \Psi_{i}) \} < \infty \,.
%$$
%}
%\medskip
%
%It is a simple exercise to check that all conditions are satisfied by
%the time-$\delta$ chain $\{\mathbf{x}_{n\delta}\}_{n = 0}^{\infty}$
%for any $\delta > 0$.
%
%\end{proof}
%}
%%%%%%%%%%%%%%%%%%%%%%%%%%%%%%%%%%%%%
\subsection{Alternate proof of Theorem 3} \ 

\medskip
As pointed out by one of our reviewers, Theorem 3 also follows from 
Theorem 4.1 in \cite{hairer2010convergence}. We thank him/her for
pointing us to this result. Below we recall the statement of it, 
and then show how to use it to deduce Theorem 3. 

\medskip
%
%
%As pointed out by the reviewer of this paper, there are some new developments of the
%coupling technique that do not explicitly use the splitting technique
%in Section 3.2 {\bf (A)} \cite{hairer2010convergence,
%  douc2009subgeometric, douc2004practical}. In particular, we
%acknowledge Theorem 4.1 in \cite{hairer2010convergence}, which
%essentially covers our case. Below, we make an alternative proof of the main result by
%using Theorem 4.1 in \cite{hairer2010convergence}. 
%
%
%
Let $\Psi_{t}$ be a strong Markov chain on a metric space $X$ with
infinitesimal generator $\mathcal{L}$ and associated semigroup
$\mathcal{P}_{t}$. The following result of sub-geometric rates of
convergence holds. 

\begin{thm}[Theorem 4.1 of \cite{hairer2010convergence}]
Assume $\Psi_{t}$ has a cadlag modification and $\mathcal{P}_{t}$ is
Feller. Assume furthermore that there exists a continuous function $V: X
\mapsto [1, \infty )$ with pre-compact sublevel sets such that 
$$
  \mathcal{L} V \leq K - \phi(V)
$$
for some constant $K$ and for some strictly concave function $\phi:
\mathbb{R}^{+} \rightarrow \mathbb{R}^{+}$ with $\phi(0) = 0$ and
increasing to infinity. In addition, we assume that sublevel sets of
$V$ are ``small'' in the sense that for every $C > 0$ there exists
$\alpha > 0$ and $T > 0$ such that 
$$
  \| \mathcal{P}_{T}(x, \cdot) - \mathcal{P}_{T}(y, \cdot) \|_{TV}
  \leq 2(1 - \alpha)
$$
for every $(x,y)$ such that $V(x) + V(y) \leq C$. Then
\begin{itemize}
  \item There exists a unique invariant measure $\mu$ for $\Psi_{t}$
    and $\mu$ is such that 
$$
  \int_{X} \phi(V(x)) \mu( \mathrm{d}x) \leq K
$$
\item Let $H_{\phi}$ be the function defined by
$$
  H_{\phi}(u) = \int_{1}^{u} \frac{\mathrm{d}s}{\phi(s)} \,.
$$
Then, there exists a constant $C$ such that for every $x, y \in X$,
one has the bounds
$$
  \| \mathcal{P}_{t}(x, \cdot) - \mathcal{P}_{t}(y, \cdot) \|_{TV}
  \leq C \frac{V(x) + V(y)}{H^{-1}_{\phi}(t)}\ .
$$
\end{itemize}
\end{thm}

The proof of Theorem 4.1 uses a different coupling that bypasses the 
explicit splitting of the Markov chain, and the Lyapunov function is lifted to $X
\times X$. Similar estimates of hitting times as in Lemma 3.5 and 3.6 are
also ingredients in this proof. 

\begin{proof}[{\bf Proof of Theorem 3 using Theorem 4.1}]

It is a simple exercise to check that (1) $\mathbf{x}_{t}$ is a strong
Markov process with an infinitesimal generator $\mathcal{G}$, and (2) $\mathbf{x}_{t}$ is a Feller
process with cadlag sample paths. 

Let $V(\mathbf{x}) = V_{\alpha}( \mathbf{x})$ be the same Lyapunov
function used before. (One may multiply $V$ by a constant to make its
minimum be greater than $1$, if necessary.) We have
$$
  \mathcal{G}V( \mathbf{x}) = \sum_{i = 1}^{m} Q_{i} \,,
$$
where
$$ 
Q_i \ = \ \sqrt{x^{i}_{0}} \left \{
    \int_{0}^{1} \left[ ( x_{0}^{i}(1 - u^{2}) + y_{0} )^{-2\alpha} +
    (x_{0}^{i} u^{2})^{-\alpha} \right] \mathrm{d}u - \left[(x_{0}^{i})^{-2\alpha} +
    y_{0}^{-\alpha} \right] \right \}\ .
$$

Therefore it follows from Lemma 2.2 that there exist constants
$\epsilon_{0} > 0$ and $c^{*} > 0$ such that
$$
  \mathcal{G}V( \mathbf{x}) \leq - c^{*}V^{\beta}( \mathbf{x})
$$
for every $\mathbf{x} \in B$, where $\beta = 1 - (4\alpha)^{-1}$ and
$$
    B = \left \{ \mathbf{x} \in \mathbf{\Delta} \,|\, y < \epsilon_{0}, \ \mbox{ or } \ 
  x^{i} < 4^{-\frac{1}{2\alpha}} \epsilon_{0} \ \mbox{ for some } \  i
  \in\{ 1, \dots, m \} \right \}\ .
$$
Let 
$$
  K = \sup_{\mathbf{x} \in \mathbf{\Delta}\setminus B} \sum_{i = 1}^{m} Q_{i}(
  \mathbf{x}) \,.
$$
It is easy to check that $K < \infty$ and
$$
  \mathcal{G}V \leq K - c^{*} V^{\beta} \,.
$$

\medskip

It remains to check that the sublevel sets of $V$ are
``small''. Let $A = A_{C}$ be the sublevel set $\{ V \leq C \}$. By the
same proof as in Proposition 3.1, for any $t > 0$ and $C > \min V(
\mathbf{x})$, there exists a constant $\eta = \eta(C, t)$ such that for every $\mathbf{x} \in A$, 
$$ 
  P^{t}(\mathbf{x}, \cdot) \geq \eta U_{A}(\cdot) \,,
$$
where $U_{A}$ is the uniform probability measure on $A$. This implies 
$$
  \| P^{t}(\mathbf{x}, \cdot) - P^{t}(\mathbf{y}, \cdot) \|_{TV} \leq
  2(1 - \eta) \,
$$
for each $\mathbf{x}, \mathbf{y}$ such that $V( \mathbf{x}) +
V(\mathbf{y}) \leq C$. 

\medskip

Therefore, let $\phi(x) = c^{*} x^{\beta}$, by Theorem 4.1, we have
$$
  \|P^{t}(\mathbf{x}, \cdot) - P^{t}(\mathbf{y}, \cdot)\|_{TV} \leq
  C_{0} \frac{V(\mathbf{x}) + V(\mathbf{y})}{ \left (c^{*} \int_{1}^{t} s^{(4
      \alpha)^{-1} - 1} \mathrm{d}s \right )^{-1}}  = C'_{0}
(V(\mathbf{x}) + V( \mathbf{y}) ) t^{-4\alpha}\,.
$$
for some constants $C_{0}$ and $C_{0}'$. The proof of Theorem 3 is
completed by letting $\alpha = \frac{1}{2} - \gamma/4$.
\end{proof}

\bigskip

%%%%%%%%%%%%%%%%%%%%%%%%%%%%%%%%%%%%%
%%%%%%%%%%%%%%%%%%%%%%%%%%%%%%%%%%%%
\bibliography{ref}{}
\bibliographystyle{plain}

\end{document}